\documentclass{amsart}
\newcommand{\RN}[1]{\textup{\uppercase\expandafter{\romannumeral#1}}}
\usepackage{amsmath}
\usepackage{amssymb}
\usepackage{enumerate}
\usepackage[hidelinks]{hyperref}

\usepackage{mathtools}
\usepackage[capitalise]{cleveref}
\numberwithin{equation}{section}

\newcommand\mc{\mathcal}

\newcommand\mb{\mathbb}
\crefname{equation}{}{}
\newtheorem{theorem}{Theorem}[section]
\newtheorem{lemma}[theorem]{Lemma}
\newtheorem{proposition}[theorem]{Proposition}
\newtheorem{corollary}[theorem]{Corollary}

\theoremstyle{definition}

\newtheorem{definition}[theorem]{Definition}

\theoremstyle{remark}
\newtheorem{remark}[theorem]{Remark}
\usepackage{epigraph}
\usepackage{comment}
\usepackage{tikz}
\usepackage{tikz-cd}

\usepackage{enumitem}
 \usepackage{url}

\title{The Degeneracy Loci for Smooth Moduli of Sheaves}
\author{Yu Zhao}
\address{Beijing Institute of Technology, Haidian, Beijing, China}
\email{zy199402@live.com}

\begin{document}

\begin{abstract}
  Let $S$ be a smooth projective surface over $\mathbb{C}$. Under certain technical assumptions, we prove that the degeneracy locus of the universal sheaf over the moduli space of stable sheaves is either empty or an irreducible Cohen–Macaulay variety of the expected dimension; we also give a criterion for when the degeneracy locus is nonempty. This result generalizes the work of Bayer, Chen, and Jiang \cite{bayer2024brill} for the Hilbert scheme of points on surfaces.

  The above statement is a special case of a more general phenomenon: for a two-term complex of locally free sheaves, the geometry of the degeneracy locus is closely related to the geometry of Grassmannians.
\end{abstract}

\maketitle
\section{Introduction}

\subsection{The Brill-Noether theory on the Moduli Space of Stable Sheaves}
Given a pair $M,N$ of moduli spaces of sheaves or complexes over a smooth variety, a fundamental problem is to study the Brill–Noether loci in $M\times N$ determined by the dimension of the space of morphisms between the corresponding objects. When $N$ is a complex algebraic surface $S$ and $M\cong S^{[n]}$ is the Hilbert scheme of ideals $I\subset \mc{O}_S$ of colength $n$, the $l$-th Brill–Noether locus is defined as 
\begin{equation*}
  \{(I,x)\in S^{[n]}\times S\mid \dim \mathrm{Hom}(I,\mb{C}_x)\geq l+1\}, \quad l\in \mb{Z}_{\geq 0}.
\end{equation*}
which is the same as the \textbf{degeneracy locus} of the universal ideal sheaf 
\begin{equation*}
  \mathrm{BN}_{l,n}:=\{(I,x)\in S^{[n]}\times S\mid \mathrm{rank}(I|_x)\geq l+1\}
\end{equation*}
by Nakayama's lemma. In \cite{bayer2024brill}, Bayer, Chen, and Jiang proved that $\mathrm{BN}_{l,n}$ is an irreducible Cohen–Macaulay variety of dimension $2n+2-l(l+1)$ if $2n-l(1+l)\geq 0$, and is otherwise empty. The Hilbert scheme of points on a surface can be viewed as the moduli space of stable sheaves of rank $1$. In this paper, we generalize the above result to the moduli space of stable sheaves of all ranks $>1$.

We describe our main result through the degeneracy locus of the universal sheaf: given a coherent sheaf $\mathcal{F}$ over a variety $X$, the degeneracy locus $\overline{D}_{\mathcal{F},l}$ is defined as the closed subvariety
\begin{equation*}
  \overline{D}_{\mathcal{F},l} := \{ p \in X \mid \mathrm{rank}(\mathcal{F}|_p) \geq l \}, \quad l \geq 0,
\end{equation*}
We fix $c_{1} \in \mathrm{NS}(S)$ and an ample line bundle $H$. For any $n \in H^{4}(S, \mathbb{Z}) \cong \mathbb{Z}$, let $\mathcal{M}^{s}_{H}(r, c_{1}, n)$ be the moduli space of Gieseker $H$-stable coherent sheaves with rank $r$, first Chern class $c_{1}$, and second Chern class $n$. Let $\mc{U}_n$ be the universal sheaf over $\mc{M}_H^s(r,c_1,n)\times S$.

\begin{theorem}
  \label{theorem1.1}
  We consider the following assumptions, which hold for a generic ample divisor when the surface \( S \) is a del Pezzo, K3, or abelian surface:
  \begin{align}
    \label{assumA}
    \gcd(r, c_1 \cdot H) = 1, \qquad   K_S \cong \mathcal{O}_S \text{ or } K_S \cdot H < 0,
  \end{align}
where $K_S$ is the canonical line bundle of $S$. If \cref{assumA} holds, then we have
\begin{enumerate}
  \item There exists $n_0\in \mb{Z}$ such that for any $n\in \mb{Z}$, $l\in \mb{Z}_{\geq 0}$, $\overline{D}_{\mathcal{U}_{n},l}$ is nonempty if and only if 
  \begin{equation*}
    g(n,l):=n-n_0- \frac{l(l-r)-t(t-r)}{2r}\geq 0.
\end{equation*}
where $t$ is the remainder of $l$ divided by $r$. In particular, $\mc{M}_H^s(r,c_1,n)$ is nonempty if and only if $n\geq n_0$. 
\item Let $\beta := (1 - r) c_1^2 - (r^2 - 1) \chi(\mathcal{O}_S) + h^1(\mathcal{O}_S)$. If $g(n,l)\geq 0$, the degeneracy locus $\overline{D}_{\mathcal{U}_{n},l}$ is an irreducible Cohen-Macaulay variety of dimension 
  \begin{equation*}
    \mathrm{vdim}_{n,l} := \beta + 2rn -\max\{0, l(l-r)\}+2,
  \end{equation*}   
\end{enumerate}
\end{theorem}
\begin{remark}
  \cref{theorem1.1} (and our proof) can be generalized to the case $r=1$ with $n_0=0$. Thus $g(n,l)=n-\frac{l(l-1)}{2}$, which is exactly the main theorem of \cite{bayer2024brill}. When $r>1$, Yoshioka \cite{yoshioka1999irreducibility,yoshioka2001moduli} proved that $n_0=\lfloor\tfrac{-\beta}{2r}\rfloor$ for a K3 surface with a generic ample divisor, and $n_0=\lfloor\tfrac{2-\beta}{2r}\rfloor$ for an abelian surface. When $S=\mathbb{P}^2$, the value of $n_0$ follows from Theorem B of \cite{ASENS_1985_4_18_2_193_0}. An upper bound for $n_0$ was given by Li–Qin \cite{Li-Qin}.

  We have an argument similar to \cref{theorem1.1} for the moduli space of framed sheaves on $\mathbb{P}^{2}$, where $n_0=0$. It is summarized in \cref{sec:appendixA}.
\end{remark}

\subsection{The Degeneracy Locus and the Grassmannian}
Different from \cite{bayer2024brill}, we prove a much more general argument and apply it to the universal sheaf. Let $X$ be an irreducible, locally complete intersection variety over an algebraically closed field, and consider a morphism of locally free sheaves over $X$
\begin{equation*}
  V \xrightarrow{\sigma} W \text{ with its dual } \sigma^{\vee}: W^{\vee} \to V^{\vee}.
\end{equation*}
Given any $l \geq 0$, we consider the $l$-th degeneracy locus $\overline{D}_{\sigma,l} := \overline{D}_{\mathrm{coker}(\sigma),l}$ as the closed subvariety consisting of points $x$ where $\mathrm{rank}(\mathrm{coker}(\sigma|_x)) \leq l$. We also consider the Grassmannian (in Grothendieck's notation) $\mathrm{Gr}_X(\sigma,l):=\mathrm{Gr}_{X}(\mathrm{coker}(\sigma),l)$, which parametrizes the quotients of $\mathrm{coker}(\sigma)$ by locally free sheaves of rank $l$. Let $e:=\mathrm{rank}(W)-\mathrm{rank}(V)$. In this paper, we prove that 
\begin{theorem}
  \label{thm:A}
  For any $l\geq 0$, either $\overline{D}_{\sigma,l}$ (resp. $\mathrm{Gr}_X(\sigma,l)$, $\mathrm{Gr}_X(\sigma^{\vee},l)$) is empty, or \cref{eq:1} (resp. \cref{eq:2}, \cref{eq:3}) holds:
  \begin{align}
    \label{eq:1} 
    \dim(\overline{D}_{\sigma,l}) \geq \dim(X)-\max\{0, l(l - e)\},  \\
\label{eq:2}
  \dim(\mathrm{Gr}_{X}(\sigma,l)) \geq \dim(X) - l(l - e), \\
  \label{eq:3}
  \dim(\mathrm{Gr}_X(\sigma^\vee,l))\geq \dim(X)-l(l+e). 
  \end{align}
  
  If one of the above three inequalities is an equality for all $l\geq 0$, the other two inequalities are equalities for all $l\geq 0$, and moreover
    \begin{enumerate}
    \item  For any $l\geq 0$, the variety $\overline{D}_{\sigma,l}$ is empty (resp. irreducible) if and only if $\mathrm{Gr}_X(\sigma,l)$ is empty (resp. irreducible). 
    \item The degeneracy locus $\overline{D}_{\sigma,l}$  is irreducible and Cohen-Macaulay for all $l\geq 0$.
    \item For any $l_-,l_+\geq 0$, the incidence variety 
\begin{equation*}
  \mathrm{Inc}_{X}(\sigma, l_{-}, l_{+}) := \mathrm{Gr}_{X}(\mathrm{coker}(\sigma^{\vee}), l_{-})\times_X\mathrm{Gr}_{X}(\mathrm{coker}(\sigma), l_{+}).
\end{equation*}
 is a locally complete intersection variety of dimension
\begin{equation*}
 d_{l_-,l_+}=\mathrm{dim}(X)-l_-(l_-+e)-l_+(l_+-e)+l_-l_+.
\end{equation*}
  \end{enumerate}
\end{theorem}
  
\begin{remark}
  In the language of \textbf{derived algebraic geometry}, the Grassmannian $\mathrm{Gr}_X(\sigma,d)$ and the incidence variety $\mathrm{Inc}_X(\sigma,l_-,l_+)$ can be endowed with a derived structure, as in \cite{jiang2022grassmanian}. Thus, in the language of derived algebraic geometry, \cref{thm:A} states that if $\mathrm{Gr}_X(\sigma,l)$ is classical for all $l\geq 0$, then $\mathrm{Gr}_X(\sigma^\vee,l)$ and $\mathrm{Inc}_X(\sigma,l_-,l_+)$ are also classical for all $l,l_-,l_+\geq 0$. Our proof of \cref{thm:A} also relies on computing the virtual dimensions of these derived schemes. 

  We note a difference in notation compared to \cite{jiang2022grassmanian}, where $\mathrm{Gr}_X(\sigma^{\vee},l)$ is denoted by $\mathrm{Gr}_X(\sigma^{\vee}[1],l)$. The reason is that we regard $\sigma$ as a morphism of locally free sheaves instead of a two-term complex. \cref{thm:A} can also be generalized to the case where $\sigma$ is a two-term complex of locally free sheaves (locally on $X$) without difficulty.
\end{remark}

According to \cite{jiang2022grassmanian}, if $X$ is a quasi-smooth variety over $\mathbb{C}$, then $\mathrm{Inc}_{X}(\sigma, l_{-}, l_{+})$ is also quasi-smooth, which induces a perfect obstruction theory in the sense of Li-Tian \cite{MR1467172} or Behrend-Fantechi \cite{MR1437495}, thereby inducing a virtual fundamental class 
\begin{equation*}
  [\mathrm{Inc}_{X}(\sigma, l_{-}, l_{+})]^{\mathrm{vir}} \in \mathrm{CH}_{d_{l_-,l_+}}(\mathrm{Inc}_{X}(\sigma, l_{-}, l_{+})).
\end{equation*}
Now we assume that all the inequalities in \cref{thm:A} are equalities. Then any incidence variety $\mathrm{Inc}_{X}(\sigma, l_{-}, l_{+})$ is a locally complete intersection variety of dimension $d_0$, thereby recovering an earlier argument by the author in \cite{zhao2024hilbertschemesblowingups}:

\begin{proposition}[Proposition 4.6 of \cite{zhao2024hilbertschemesblowingups}]
\label{prop:prop}
  We have
  \begin{equation*}
    [\mathrm{Inc}_{X}(\sigma, l_{-}, l_{+})]^{\mathrm{vir}} = [\mathrm{Inc}_{X}(\sigma, l_{-}, l_{+})] \in \mathrm{CH}_{d_{0}}(\mathrm{Inc}_{X}(\sigma, l_{-}, l_{+})).
  \end{equation*}
\end{proposition}

\subsection{Notations} In this paper, we assume that the dimension of the empty set is an arbitrary integer, so that any equalities and inequalities involving the dimension of a variety continue to hold even when the variety is empty.

\subsection{Acknowledgement}  While preparing this paper, the author was reminded by Qingyuan Jiang that a similar argument to \cref{thm:A} was given in his note \cite{JiangNotes}. The author would like to thank Qingyuan Jiang for many helpful discussions.

\section{Degeneracy Theory for Two-term Complexes}
\label{sec2}
In this section, let $X$ be a nonempty locally complete intersection variety over an algebraically closed field, and let $\sigma:V\to W$ be a morphism of locally free sheaves over $X$. Let $e:=\mathrm{rank}(W)-\mathrm{rank}(V)$, and let $\sigma^{\vee}:W^{\vee}\to V^{\vee}$ be the dual of $\sigma$.

\begin{definition}[Section 14.4 of \cite{MR0732620}]
For any integer $l$, we define the $l$-th degeneracy locus as
\begin{equation*}
    \overline{D}_{\sigma,l} := \{ x\in X \mid \mathrm{rank}(\mathrm{coker}(\sigma|_{k(x)})) \geq l \},
\end{equation*}
which is a closed subvariety of $X$, where $k(x)$ denotes the residue field of $x$. We also define the locally closed subvariety
\begin{equation*}
    D_{\sigma,l} := \overline{D}_{\sigma,l} - \overline{D}_{\sigma,l+1} = \{ x \mid \mathrm{rank}(\mathrm{coker}(\sigma|_{k(x)})) = l \}.
\end{equation*}
\end{definition}
By the definition of degeneracy loci, we have $\overline{D}_{\sigma,k} = \overline{D}_{\sigma^{\vee},k-e}$ and 
\begin{equation*}
    \overline{D}_{\sigma,k} \cong X \quad \mathrm{if} \quad k \leq \max\{0, e\}, \quad D_{\sigma,k} \cong \emptyset \quad \mathrm{if} \quad k < \max \{0, e\}.
\end{equation*}

\begin{proposition}[Theorem 14.4 (b)(c) of \cite{MR0732620}.]
\label{prop:B4}
For any integer $k\geq \max\{0,e\}$, we have
\begin{equation*}
    \dim(\overline{D}_{\sigma,k})\geq \dim(X)-k(k-e).
\end{equation*}
In particular, for any $k\geq \max\{0,e\}$, if the  inequality is an equality, then $\overline{D}_{\sigma,k}$ (resp. $D_{\sigma,k}$) is Cohen-Macaulay (resp. a locally complete intersection variety).
\end{proposition}

\begin{lemma}
\label{lemB8}
For any integer $l\geq 0$, we have
\begin{equation}
  \label{eq:88}
    \dim(\mathrm{Gr}_{X}(\sigma,l)) \geq \dim(X) - l(l - e), \quad \dim(\mathrm{Gr}_X(\sigma^\vee,l))\geq \dim(X)-l(l+e).
\end{equation}
Moreover, the first inequality of \cref{eq:88} is an equality for all $l\geq 0$ if and only if the second inequality is an equality for all $l\geq 0$, which is also equivalent to the condition
\begin{equation}
  \label{eq:89}
    \dim(\overline{D}_{\sigma,l})=\dim(X)-l(l - e), \text{ for all } l\geq \max\{0,e\}.
\end{equation}
If the above equivalent conditions hold, for all $l\geq 0$, $\overline{D}_{\sigma,l}$ is empty (resp. irreducible) if and only if $\mathrm{Gr}_X(\sigma,l)$ is empty (resp. irreducible). 
\end{lemma}
\begin{proof}
  We only prove the inequalities about $\mathrm{Gr}_X(\sigma,l)$, as the rank of $\sigma^{\vee}$ is $-e$ and we can apply the same argument to $\sigma^{\vee}$. By \cite{jiang2022grassmanian}, $\mathrm{Gr}_X(\sigma,l)$ is endowed with the structure of a quasi-smooth derived scheme and its virtual dimension is $\dim(X)+l(e-l)$. As the dimension of a quasi-smooth scheme is always greater than or equal to its virtual dimension, we have \cref{eq:88}. 

  Now assume that the first inequality of \cref{eq:88} is an equality for all $l\geq 0$. Let $pr_{\sigma,l}$ be the projection from $\mathrm{Gr}_X(\sigma,l)$ to $X$. Then the fiber over $D_{\sigma,k}$ is a $\mathrm{Gr}(k,l)$-bundle; in particular, over $D_{\sigma,l}$ the fiber is $\mathrm{Gr}(l,l)$ (a point). Thus, we have
  \begin{equation*}
    \dim(D_{\sigma,l})\leq \dim(X)-l(l-e)
  \end{equation*}
  for all $l\geq \max\{0,e\}$ if it is not empty. 
  As $\dim(X)-l(l-e)$ is a strictly decreasing function of $l$ when $l\geq \max\{0,e\}$, we have \cref{eq:89}. 

Next we assume that \cref{eq:89} holds for all $l\geq 0$. We notice that the fiber of $D_{\sigma,k}$ is empty when $k<\max\{l,e\}$, and has dimension 
  \begin{equation*}
    \dim(X)-k(k-e)+l(k-l)
  \end{equation*}
  which is a strictly decreasing function of $k$ when $k\geq \max\{l,e\}$. Hence it is achieved at the maximal value at $k=e$ or $k=l$, which are both $\dim(X)+l(e-l)$. Thus the inequalities in \cref{eq:88} are equalities for all $l\geq 0$. 

  Finally, we notice that $\overline{D}_{\sigma,l}$ and $\mathrm{Gr}_X(\sigma,l)$ are both nonempty when $0\leq l\leq\max\{0,e\}$. Now assume that \cref{eq:89} holds for all $l\geq \max\{0,e\}$. Then when $l\geq \max\{0,e\}$, $D_{\sigma,l}$ is dense in $\overline{D}_{\sigma,l}$. On the other hand, as \begin{equation*}
    \dim(X)-k(k-e)+l(k-l), \quad k\geq \max\{l,e\}
  \end{equation*}
  achieves its maximal value at $k=l$, the union of fibers over $D_{\sigma,k}$ in $\mathrm{Gr}_X(\sigma,l)$ for all $k> l$ has strictly smaller dimension than $\dim(X)+l(e-l)$. Thus, $D_{\sigma,l}$ is also dense in $\mathrm{Gr}_X(\sigma,l)$. Hence
  $\mathrm{Gr}_X(\sigma,l)$ is empty if and only if $\overline{D}_{\sigma,l}$ is empty.
\end{proof}

\begin{proof}[Proof of \cref{thm:A}]
  It remains to prove that if \cref{eq:89} holds for all $l\geq \max\{0,e\}$, then $\mathrm{Inc}_X(\sigma,l_-,l_+)$ is a locally complete intersection variety of dimension 
  \begin{equation*}
   d_0= \dim(X) - l_-(l_- + e) - l_+(l_+ - e)+l_-l_+.
  \end{equation*}
  The reason is that its fiber over $D_{\sigma,k}$ is empty if  $k < \max\{l_- + e, l_+\}$
  and its dimension is 
  \begin{equation*}
    \dim(X) - k(k - e) + l_-(k -e- l_-) + l_+(k - l_+)
  \end{equation*}
  which is strictly decreasing when $k \geq \max\{l_- + e, l_+\}$. Thus the maximal value is achieved at $k = \max\{l_- + e, l_+\}$, which is exactly $d_0$. The irreducibility of $\mathrm{Inc}_X(\sigma,l_-,l_+)$ follows from the fact that it is endowed with a quasi-smooth derived scheme structure in \cite{jiang2022grassmanian} and its virtual dimension is also $d_0$ (see Appendix A of \cite{ddd} for more details).    
\end{proof}

\section{Degeneracy Loci on the Smooth Moduli Surface of Coherent Sheaves over a Surface}
\label{sec3}

  In this section, we prove \cref{theorem1.1}. We always assume \cref{assumA} for the remainder of this section. We first recall basic facts about the moduli space of stable sheaves. 
\begin{proposition}\label{existence}
  There exists a unique integer $n_0\in \mb{Z}$ such that $\mathcal{M}_H^{s}(r, c_1, n)$ is not empty if and only if $n\geq n_0$. 
\end{proposition}
\begin{proof}
 By the Bogomolov inequality, $\mc{M}_{H}^s(r,c_1,n)$ is empty for $n< \tfrac{1}{4}c_1^2$. On the other hand, by Theorem 1.1 of \cite{Li-Qin}, there exists a constant $\alpha$ such that $\mc{M}_H^s(r,c_1,n)$ is nonempty for all $n\geq \alpha$. Thus there exists a smallest integer $n_0$ such that $\mc{M}_H^s(r,c_1,n)$ is nonempty. Let $U\in \mc{M}_H^s(r,c_1,n_0)$ be a stable sheaf. Set $U_{n_0}:=U$, and inductively construct a sequence of coherent sheaves $U_{n_{0}}\supset U_{n_{0}+1}\supset U_{n_0+2},\ldots$ such that $U_{k}/U_{k+1}$ is a length-1 skyscraper sheaf for all $k\geq n_0$. By Proposition 5.5 of \cite{negut2017shuffle}, all $U_k$ are stable coherent sheaves.
\end{proof}
\begin{proposition}[Proposition 2.10, 2.14 of \cite{neguct2018hecke} and Lemma 3.3 of \cite{koseki2021categorical}]
  \label{const}
   When $n\geq n_0$, $\mathcal{M}_H^{s}(r, c_1, n)$
is a smooth projective variety of dimension $\beta + 2rn$. Moreover, the universal sheaf \( \mathcal{U}_n \) has a resolution of the form 
  \begin{equation*}
    0\to \mc{V}_n\xrightarrow{\sigma_n} \mc{W}_n\to \mc{U}_n\to 0,
  \end{equation*}
  where $\sigma_n$ is a morphism between the locally free sheaves $\mc{V}_n$ and $\mc{W}_n$ over $\mc{M}_H^{s}(r, c_1, n)\times S$. 
\end{proposition}

Given a closed point $o\in S$, let $\sigma_{n,o}$ be the restriction of $\sigma_n$ to $\mc{M}_H^s(r,c_1,n)\times \{o\}$, where $\sigma_{n}$ follows from \cref{const}.
\begin{proposition}
  \label{prop:1}
  Given integers $l,n$, we have 
   \begin{equation*}
    \mathrm{Gr}_{\mc{M}^{s}_{H}(r,c_1,n)}(\sigma_{n,o}^{\vee},l)\cong \mathrm{Gr}_{\mc{M}^{s}_{H}(r,c_1,n-l)}(\sigma_{n-l,o},l),
  \end{equation*}
  which is either empty or a smooth variety of dimension $\beta+2rn-l(l+r)$.
\end{proposition}
\begin{proof}
  It is actually a corollary of the study of stable perverse coherent sheaves on blowups of surfaces by Nakajima–Yoshioka \cite{1050282810787470592}. 
  
  Let $\hat{S}:=\mathrm{Bl}_oS$ and $p:\hat{S}\to S$ be the projection map.  Let $C:=p^{-1}(o)\cong \mb{P}^{1}$ be the exceptional divisor and $\mc{O}_C(-1):=\mc{O}(C)|_C$.  Let $[C]\in H^2(\hat{S},\mb{Z})$ be the fundamental class of $C$ in $\hat{S}$.

A stable perverse coherent sheaf $E$ on $\hat{S}$ with respect to $H$ is defined as a coherent sheaf such that
    \begin{equation*}
      \operatorname{Hom}(E,\mc{O}_{C}(-1))=0,\quad \operatorname{Hom}(\mc{O}_{C},E)=0
    \end{equation*}
    and $p_{*}E$ is Gieseker stable with respect to $H$. For any $n,l\in \mb{Z}$, we define $M^{0}(n,l)$ as the moduli space of perverse coherent sheaves with the total Chern class 
    \begin{equation*}
  v_{n,l}:=(r,p^*c_1-l[C],n-\frac{l^2+l}{2}).
\end{equation*}
Its corresponding Chern character is
\begin{equation*}
  p^*(r,c_1,\frac{1}{2}c_1^2-n)-l(ch(\mc{O}_C(-1))).
\end{equation*}

By Theorem 4.1 of \cite{1050282810787470592} (see also Theorem 4.1 of \cite{koseki2021categorical}), we have
 \begin{equation*}
    M^{0}(n,l)\cong \mathrm{Gr}_{\mc{M}^{s}_{H}(c)}(\sigma_{n,o}^{\vee},l),\quad  M^{0}(n,l)\cong \mathrm{Gr}_{\mc{M}^{s}_{H}(c-l[pts])}(\sigma_{n-l,o},l).
  \end{equation*}
By Lemma 3.22 of \cite{1050282810787470592} (see also Lemma 3.3 of \cite{koseki2021categorical}), $M^{0}(n,l)$ is either empty or a smooth variety of dimension $\beta+2rn-l(l+r)$. Hence we finish the proof of \cref{prop:1}.
\end{proof}

\begin{proof}[Proof of \cref{theorem1.1}] We choose a closed point $o\in S$. We first prove that $\overline{D}_{\sigma_{n,o},l}$ is nonempty if and only if $g(n,l)\geq 0$. When $l\geq r$, we notice that by \cref{thm:A} and \cref{prop:1},
  $\overline{D}_{\sigma_{n,o},l} = \overline{D}_{\sigma_{n,o}^{\vee},l-r}$ is not empty 
 if and only if $\overline{D}_{\sigma_{n-l+r,o},l-r}$ is not empty. When $l\geq r$, we also have 
\begin{equation*}
  g(n,l)=g(n-l+r,l-r).
\end{equation*}
Thus we can reduce to the case that $0\leq l<r$. In this case, $g(n,l)=n-n_0$ and $\overline{D}_{\sigma_{n,o},l}=\mc{M}_H^s(r,c_1,n)$, which is nonempty if and only if $n-n_0\geq 0$. Thus, we prove the argument. 

Next, we prove the argument that $\overline{D}_{\sigma_{n,o},l}$ is an irreducible Cohen-Macaulay variety of dimension $\beta+2rn-\max\{0,l(l-r)\}$ if $g(n,l)\geq 0$. It also directly follows from \cref{thm:A} and \cref{prop:1}.

Now we consider $\overline{D}_{\mc{U}_n,l}=\overline{D}_{\sigma_{n},l}$. We know that $\overline{D}_{\sigma_{n},l}$ is nonempty if and only if $\overline{D}_{\sigma_{n,o},l}$ is nonempty for some closed point $o\in S$. Hence it is equivalent to $g(n,l)\geq 0$. Now assume that $g(n,l)\geq 0$. Regarding $\overline{D}_{\sigma_{n},l}$ as a scheme over $S$, for any closed point $o\in S$, its fiber is $\overline{D}_{\sigma_{n,o},l}$ and has dimension $\beta+2rn-\max\{0,l(l-r)\}$. Hence we have 
  \begin{equation*}
    \dim(\overline{D}_{\sigma_{n},l}) = \beta+ 2rn - \max\{0, l(l-r)\}+2.
  \end{equation*}
   Thus, we finish the proof of \cref{theorem1.1}.
\end{proof}

\appendix
\section{Moduli Space of Framed Coherent Sheaves on the Plane and ADHM Construction}
\label{sec:appendixA}
Given integers $(r,n) \in \mathbb{Z}_{>0} \times \mathbb{Z}_{\geq 0}$, let
$\mathcal{M}_{fr}(r,n)$ be the moduli space of framed sheaves on the projective plane. It can also be interpreted as a Nakajima quiver variety: we consider two vector spaces $V$ and $W$ with $\dim(V)=n$ and $\dim(W)=r$. The ADHM datum for this setup is given by
\begin{equation*}
  \{(X,Y,i,j) \in \operatorname{Hom}(V,V) \oplus \operatorname{Hom}(V,V) \oplus \operatorname{Hom}(W,V) \oplus \operatorname{Hom}(V,W)\}
\end{equation*}
subject to the equation
\begin{equation*}
  [X,Y] + ij = 0,
\end{equation*}
An ADHM datum $(X,Y,i,j)$ is called stable if there is no non-zero subspace $V' \subset V$ closed under $X$ and $Y$ and contained in the kernel of $j$. Let $\mathrm{M}^{s}(V,W)$ be the space of stable ADHM data. The general linear group $\mathrm{GL}(V)$ acts on $\mathrm{M}^{s}(V,W)$ by
\begin{equation*}
  g \cdot (X,Y,i,j) = (g X g^{-1}, g Y g^{-1}, g i, j g^{-1}).
\end{equation*} 
\begin{theorem}[Chapter 2 of \cite{MR1711344}]We have
\begin{equation*} 
  \mathcal{M}_{fr}(r,n) \cong \mathrm{M}^{s}(V,W)/\mathrm{GL}(V),
\end{equation*}
Moreover, $\mathcal{M}_{fr}(r,n)$ is a non-empty smooth variety of dimension $2rn$.
\end{theorem}

The vector spaces $V$ and $W$ induce tautological locally free sheaves on $\mc{M}_{fr}(r,n)$, denoted by $\mathcal{V}_n$ and $\mathcal{W}$. We consider the two-term complex
\begin{equation*}
  \mathcal{V}_n \xrightarrow{(X - x \cdot \text{id}) \oplus (Y - y \cdot \text{id}) \oplus j} \mathcal{V}_n \oplus \mathcal{V}_n \oplus \mathcal{W} \xrightarrow{(Y - y \cdot \text{id}) \oplus -(X - x \cdot \text{id}) \oplus i} \mathcal{V}_n,
\end{equation*}
for $(x,y) \in \mathbb{A}^{2}$. As the map
\begin{equation*}
  (X - x \cdot \text{id}) \oplus (Y - y \cdot \text{id}) \oplus j
\end{equation*}
is injective at every closed point, its cokernel, denoted by $\mathcal{T}_{n}$, is locally free. Moreover, we have a morphism of locally free sheaves
\begin{equation*}
  \sigma_{n}: \mathcal{V}_n \to \mathcal{T}_n,
\end{equation*}
where the (virtual) rank of $\sigma_n$ is $-r$. Let $\mc{U}_n$ be the universal sheaf $\mathcal{U}$ on $\mathcal{M}_{fr}(r,n) \times \mathbb{A}^{2}$. Then we have 
\begin{equation*}
  \mc{U}_n \cong \mathrm{coker}(\sigma_n^{\vee}).
\end{equation*}
Let $\sigma_{n,0}$ denote the restriction of $\sigma_n$ to $\mathcal{M}_{fr}(r,n) \times \{(0,0)\}$, obtained by setting $x = y = 0$.

The Grassmannian of $\mathcal{U}_n$ also has a quiver-like description due to 
 Nakajima–Yoshioka \cite{NYP1}. Let $V_{1}$, $V_{2}$, and $W$ be vector spaces of dimensions $n_{1}$, $n_{2}$, and $r$, respectively. We consider the datum $(B_{1}, B_{2}, d, i, j)$ in
\begin{equation*}
  \operatorname{Hom}(V_{1}, V_{2}) \oplus \operatorname{Hom}(V_{2}, V_{1}) \oplus \operatorname{Hom}(W, V_{1}) \oplus \operatorname{Hom}(V_{2}, W),
\end{equation*}
satisfying the equation
\begin{equation*}
  B_{1} d B_{2} - B_{2} d B_{1} = 0,
\end{equation*}
A datum $(B_{1}, B_{2}, d, i, j)$ is called stable if for any subspaces \begin{equation*}
  V_{1}' \subset V_{1}, \quad  V_{2}' \subset V_{2}, \quad 0 \subset W
\end{equation*}
which are closed under the quiver representation, we must have $V_{1}' = 0$. By the stability condition, $d$ is injective, implying $n_{1} \geq n_{2}$. Let $\mathrm{M}^{s}(V_{1}, V_{2}, W)$ be the space of stable data.
\begin{theorem}[Nakajima-Yoshioka \cite{NYP1}] 
The moduli space
\begin{equation*}
  \mathcal{M}_{\text{perv}}(n_{1}, n_{2}, r) := \mathrm{M}^{s}(V_{1}, V_{2}, W) / (\mathrm{GL}(V_{1}) \times \mathrm{GL}(V_{2}))
\end{equation*}
is either an irreducible smooth variety of dimension $2n_{1}n_{2} - n_{1}^{2} - n_{2}^{2} + (n_{1} + n_{2})r$ or an empty set. Moreover, we have
\begin{align*}
  \operatorname{Gr}_{\mathcal{M}_{fr}(n_{2}, r)}(\sigma_{n_2,0}^{\vee}, n_{1} - n_{2})\cong \mathcal{M}_{\text{perv}}(n_{1}, n_{2}, r) \cong \operatorname{Gr}_{\mathcal{M}_{fr}(n_{1}, r)}(\sigma_{n_1,0}, n_{1} - n_{2})
\end{align*}
\end{theorem}

Applying \cref{thm:A}, we conclude:
\begin{corollary}
  For any integers $r>0$, $n\geq 0$, $l\geq 0$, the degeneracy locus $\overline{D}_{\mc{U}_{n},l}$ over $\mc{M}_{fr}(r,n)\times \mb{A}^{2}$ is nonempty if and only if 
  \begin{equation*}
    \tilde{g}(n,l):=n-\frac{l(l-r)-t(t-r)
    }{2r}\geq 0,
  \end{equation*}
  where $t$ is the remainder when $l$ is divided by $r$. Moreover, if $\tilde{g}(n,l)\geq 0$, then $\overline{D}_{\mc{U}_{n},l}$ is an irreducible Cohen–Macaulay variety of dimension 
  $$2rn + 2 - \max\{0, l(r-l)-t(r-t)\}.$$

\end{corollary}
\begin{proof}
  The proof is the same as the proof of \cref{theorem1.1}, which we omit here.
\end{proof}

\end{document}